\documentclass[12pt]{article}
\usepackage{color}

\usepackage{amsmath}
\usepackage{amsfonts}
\usepackage{amsthm}
\usepackage{amssymb}
\usepackage{bm}
\usepackage[applemac]{inputenc}
\usepackage{enumerate}
\usepackage{tipx}
\makeatletter

\newtheoremstyle{mystyle}{}{}{\slshape}{2pt}{\scshape}{.}{ }{} 

\newtheorem{thm}{Theorem}[section]
\newtheorem{cor}[thm]{Corollary}

\newtheorem{prop}[thm]{Proposition}
\newtheorem{lemme}[thm]{Lemma}
\newtheorem{fait}[thm]{Fact}

\newtheorem{question}[thm]{Question}

\theoremstyle{definition}
\newtheorem{defi}[thm]{Definition}
\theoremstyle{mystyle}

\theoremstyle{remark}

\newcommand{\impl}{\rightarrow}
\newcommand{\pred}{\mathbf P}

\DeclareMathOperator{\tp}{tp}
\DeclareMathOperator{\acl}{acl}
\DeclareMathOperator{\st}{st}

\title{Externally definable sets and dependent pairs}
\author{Artem Chernikov \thanks{Supported by the Marie Curie Initial Training Network in Mathematical Logic  - MALOA - From MAthematical LOgic to Applications, PITN-GA-2009-238381} \and Pierre Simon}
\begin{document}
\maketitle
\begin{abstract}
We prove that externally definable sets in first order $NIP$ theories have \emph{honest definitions}, giving a new proof of Shelah's expansion theorem. Also we discuss a weak notion of stable embeddedness true in this context. Those results are then used to prove a general theorem on dependent pairs, which in particular answers a question of Baldwin and Benedikt on naming an indiscernible sequence.
\end{abstract}
\section*{Introduction}

This paper is organised in two main parts, the first studies externally definable sets in first order $NIP$ theories and the second, using those results, proves dependence of some theories with a predicate, under quite general hypothesis. We believe both parts to be of independent interest.
A third section gives some examples of dependent pairs and relates results proved here to ones existing in the literature.

\subsubsection*{Honest definitions}

Let $M$ be a model of a theory $T$. An \emph{externally definable} subset of $M^k$ is an $X\subseteq M^k$ that is equal to $\phi(M^k,d)$ for some formula $\phi$ and $d$ in some $N\succ M$. In a stable theory, by definability of types, any externally definable set coincides with some $M$-definable set. By contrast, in a random graph for example, any subset in dimension 1 is externally definable.

Assume now that $T$ is $NIP$. A theorem of Shelah (\cite{Sh863}), generalising a result of Poizat and Baisalov in the o-minimal case (\cite{Poizat}), states that the projection of an externally definable set is again externally definable. His proof does not give any information on the formula defining the projection. A slightly clarified account is given by Pillay in \cite{Pillayexp}.

In section 1, we show how this result follows from a stronger one: existence of honest definitions. An \emph{honest definition} of an externally definable set is a formula $\phi(x,d)$ whose trace on $M$ is $X$ and which implies all $M$-definable subsets containing $X$. Then the projection of $X$ can be obtained simply by taking the trace of the projection of $\phi(x,d)$.

Combining this notion with an idea from \cite{Guingona}, we can adapt honest definitions to make sense over any subset $A$ instead of a model $M$. We obtain a property of \emph{weak stable-embeddedness} of sets in $NIP$ structures. Namely, consider a pair $(M,A)$, where we have added a unary predicate $\pred(x)$ for the set $A$. Take $c\in M$ and $\phi(x,c)$ a formula. We consider $\phi(A,c)$. If $A$ is stably  embedded, then this set is $A$-definable. Guingona shows that in an $NIP$ theory, this set is externally $A$-definable, {\it i.e.}, coincides with $\psi(A,d)$ for some $\psi(x,y)\in L$ and $d\in A'$ where $(M',A') \succ (M,A)$. We strengthen this by showing that one can find such a $\phi(x,d)$ with the additional property that $\psi(x,d)$ never lies, namely $(M',A') \models \psi(x,d) \rightarrow \phi(x,c)$. In particular, the projection of $\psi(x,d)$ has the same trace on $A$ as the projection of $\phi(x,c)$. This is the main tool used in Section 2 to prove dependence of pairs.

\subsubsection*{Dependent pairs}

In the second part of the paper we try to understand when dependence of a theory is preserved after naming a new subset by a predicate. We provide a quite general sufficient condition for the dependence of the pair, in terms of the structure induced on the predicate and the restriction of quantification to the named set.

This question was studied for stable theories by a number of people (see  \cite{CaZi} and \cite{BoBa} for the most general results). In the last few years there has been a large number of papers proving dependence for some pair-like structures, e.g. \cite{BDO}, \cite{GH1}, \cite{Box}, etc. We apologise for adding yet another result to the list. However, our approach differs in an important way from the previous ones, in that we work in a general $NIP$ context and do not make any assumption of minimality of the structure (by asking for example that the algebraic closure controls relations between points). In particular, in the case of pairs of models, we obtain that if $M$ is dependent, $N \succ M$ and $(N,M)$ is bounded (see Section 2 for a definition), then $(N,M)$ is dependent.

Those results seem to apply to most, if not all, of the pairs known to be dependent. It also covers some new cases, in particular answering a question of Baldwin and Benedikt about naming an indiscernible sequence.

\subsubsection*{The setting}

We will not make a blanket assumption that $T$ is $NIP$, so we work a priori with a general first order theory $T$ in a language $L$. We use standard notation. We have a monster model $\mathbb M$. If $A$ is a set of parameters, $L(A)$ denotes the formulas of $L$ with parameters from $A$. If $\phi(x)$ is some formula, and $A$ a subset of $\mathbb M$, we will write $\phi(A)$ for the set of tuples $a\in A^{|x|}$ such that $\phi(a)$ holds. If $A$ is a set of parameters, by $\phi(x) \rightarrow^{A} \psi(x)$, we mean that for every $a\in A$, $\phi(a) \rightarrow \psi(a)$ holds. Also $\phi(x) \rightarrow^{p(x)} \psi(x)$ stands for $\phi(x) \rightarrow^{p(\mathbb M)} \psi(x)$.

We will often consider pairs of structures. So if our base language is $L$, we define the language $L_{\mathbf{P}}$ where we add to $L$ a new unary predicate $\mathbf P(x)$.
If $M$ is an $L$-structure and $A\subseteq M$, by the pair $(M,A)$ we mean the $L_{\mathbf P}$ extension of $M$ obtained by setting $\pred(a) \Leftrightarrow a\in A$. Throughout the paper $\pred(x)$ will always denote this extra predicate.
\\

As usual $alt(\phi)$ is the maximal number $n$ such that there exists an indiscernible sequence $( a_i)_{i<n}$ and $c$ satisfying $\phi(a_i,c) \Leftrightarrow i$ is even. Standardly $\phi(x,y)$ is dependent if and only if $alt(\phi)$ is finite. For more on the basics of dependent theories see e.g. \cite{Adl}.

\subsubsection*{Acknowledgments}

We are grateful to Itay Kaplan and the referee for suggesting a number of improvements to the paper, and to Ita\"i Ben Yaacov and Manuel Bleichner for pointing out some typos and deficiencies.

\section{Externally definable sets and honest definitions}

Recall that a partial type $p(x)$ is said to be \emph{stably embedded} if any definable subset of $p(x)$ is definable with parameters from $p(\mathbb M)$. It is well known that if $p(x)$ is stable, then $p(x)$ is stably embedded (see e.g. \cite{OnPe}). We are concerned with an analogous property replacing stable by dependent.

We say that a formula $\phi(x,c)$ is $NIP$ over a (partial) type $p(x)$ if there is no indiscernible sequence $(a_i)_{i<\omega}$ of realisations of $p$ such that $\phi(a_i,c)$ holds if and only if $i$ is even. We say that $\phi(x,y)$ is $NIP$ over $p(x)$ if $\phi(x,c)$ is $NIP$ over $p(x)$ for every $c$.

The following is the fundamental observation. We assume here that we have two languages $L\subseteq L'$, and we work inside a monster model $\mathbb M$ that is an $L'$-structure. The language $L'$ could be $L_\pred$ for example.

\begin{prop}
\label{HonestLemma} Let $p(x)$ be a partial
$L'$-type and $\phi(x,c)\in L(\mathbb M)$ be $NIP$ over $p(x)$. Then for
each small $A\subseteq p(\mathbb{M})$ there is $\theta(x)\in L(p(\mathbb{M}))$
such that 

1) $\theta(x)\cap A=\phi(x,c)\cap A$

2) $\theta(x)\rightarrow^{p(x)}\phi(x,c)$

3) $\phi(x,c)\setminus\theta(x)$ does not contain any $A$-invariant
global $L$-type consistent with $p(x)$.\end{prop}
\begin{proof}
Let $q(x)\in S_{L}(\mathbb{M})$ be $A$-invariant and consistent
with $\{\phi(x,c)\}\cup p(x)$. We try to choose inductively $a_{i},b_{i}\in p(\mathbb{M})$
and $q_{i}\subseteq q$, for $i<\omega$ such that 

- $q_{i}(x)=q(x)|_{Aa_{<i}b_{<i}}$

- $a_{i}\models q_{i}(x)\cup\{\phi(x,c)\}  \cup p(x) $ (we can always find one by
assumption)

- $b_{i}\models q_{i}(x)\cup\{\neg\phi(x,c)\}  \cup p(x)$.

Assume we succeed. Consider the sequence $(d_i)_{i<\omega}$ where $d_i = a_i$ if $i$ is even and $d_i=b_i$ otherwise. It is a Morley sequence of $q$ over $A$, and as such is $L$-indiscernible. Furthermore, we have $\models \phi(d_i,c)$ if and only if $i$ is even. This contradicts $\phi(x,y)$ being $NIP$ over $p(x)$, so the construction must stop at some finite stage $i_0$. Then $q_{i_0}(x)\rightarrow^{p(x)}\phi(x,c)$
and by compactness there is $\psi_{q}(x)\in q_{i_0}$ (so $\psi_{q}\in L(p(\mathbb{M})$))
such that $\psi_{q}(x)\rightarrow^{p(x)}\phi(x,c)$. So we
see that the set of all such $\psi_{q}$'s covers the compact space
of global $L$-types invariant over $A$ and consistent with $\{\phi(x,c)\}\cup p(x)$
(so in particular all realised types of elements of $A$ such that $\phi(a,c)$). Let $(\psi_{j})_{j<n}$ be a
finite subcovering, then taking $\theta(x)=\bigvee_{j<n}\psi_{j}(x)$
does the job.\end{proof}

\begin{defi}[Externally definable set]
Let $M$ be a model, an externally definable set of $M$ is a subset $X$ of $M^k$ for some $k$ such that there is a formula $\phi(x,y)$ and $d\in \mathbb M$ with $\phi(M,d) = X$. Such a $\phi(x,d)$ is called a definition of $X$.
\end{defi}

We can now prove a form of \emph{weak stable embeddedness} for $NIP$ formulas.

\begin{cor}[Weak stable-embeddedness]\label{weakstable}
Let $\phi(x,y)$ be $NIP$. Given $(M,A)$ and $c\in M$ there are $(M',A')\succeq(M,A)$ and $\theta(x)\in L(A')$ such that $\phi(A,c)=\theta(A)$
and $\theta(x)\rightarrow^{A'}\phi(x,c)$.
\end{cor}
\begin{proof}
Notice that $\phi(x,y)$ is still $NIP$ in any expansion of the structure. In particular in the $L_\pred$-structure $(M,A)$. Now
apply Proposition \ref{HonestLemma} with $L'=L_\pred$ and $p(x)=\{\pred(x)\}$.
\end{proof}

\begin{question} Do we get uniform weak stable embeddedness ? In other words, is it possible to choose $\theta$ depending just on $\phi$, or at least just on $\phi$ and $Th(M,A)$ ?
\end{question}

\begin{cor}
Let $f: M \to M$ be an externally definable function, that is the trace on $M$ of an externally definable relation which happens to be a function on $M$. Then there is an $\mathbb{M}$-definable partial function $g: \mathbb{M} \to \mathbb{M}$ with $g|_{M}=f$.	
\end{cor}
\begin{proof}
Let $\phi(x,y;c)$ induce $f$ on $M$, $c \in N \succ M$. By Corollary \ref{weakstable} we find $(N',M') \succ (N,M)$ and $\theta(x,y) \in L(M')$ satisfying $\theta(M^2)=\phi(M^2,c)$ and $\theta(x,y) \rightarrow^{M'} \phi(x,y;c)$. As the extension of pairs is elementary and $M' \models T$, it follows that $\theta(x,y)$ is a graph of a global partial function.
\end{proof}

\begin{defi}[Honest definition]\label{honestDef}
Let $X\subseteq M^k$ be externally definable. Then an honest definition of $X$ is a definition $\phi(x,d)$ of $X$, $d\in \mathbb M$ such that:
\begin{description}
\item $\mathbb M \models \phi(x,d) \rightarrow \psi(x)$ for every $\psi(x) \in L(M)$ such that $X\subseteq \psi(M)$.
\end{description}
\end{defi}

In Section 2, we will need the notion of an honest definition \emph{over $A$} which is defined at the beginning of that section.

\begin{prop}\label{HonestDefExist}
Let $T$ be $NIP$. Then every externally definable set $X\subset M^k$ has an honest definition.
\end{prop}
\begin{proof}
Let $M\prec N$ and $\phi(x)\in L(N)$ be a definition of $X$, and let $(N',M')\succeq(N,M)$
be $|N|^{+}$-saturated (in $L_P$). Let $\theta(x)\in L(M')$ as given by Corollary \ref{weakstable}, so $(N',M')\models(\forall x\in \pred)\,\theta(x)\rightarrow\phi(x)$. If $\psi(x)\in L(M)$
with $X\subseteq\psi(M)$ then $(N',M') \models(\forall x\in \pred)\,\phi(x)\rightarrow\psi(x)$.
Combining, we get $(N',M') \models(\forall x\in \pred)\,\theta(x)\rightarrow\psi(x)$. But
since $M'\models T$
and $\theta(x),\psi(x)\in L(M')$ we have finally $M' \models \theta(x)\rightarrow\psi(x)$.
\end{proof}

We illustrate this notion with an o-minimal example inspired by \cite{Poizat}.

We let $M_0$ be the real closure of $\mathbb Q$ and let $\epsilon>0$ be an infinitesimal element. Let $M$ be the real closure of $M_0(\epsilon)$. Let $\pi$ be the usual transcendental number, and finally let $N$ be the real closure of $M(\pi)$.

\begin{lemme}
Let $0<b\in N$ be infinitesimal, then there is $n\in \mathbb N$ such that $b<\epsilon^{1/n}$.
\end{lemme}
\begin{proof}
We define a valuation $v$ on $\mathbb Q(\pi,\epsilon)$ by setting $v(x)=0$ for all $x\in \mathbb Q(\pi)$ and $v(\epsilon)=1$. We also define a valuation on $N$ with the following standard construction: let $\mathcal O\subset N$ be the convex closure of $\mathbb Q$ and $\mathfrak M$ be the ring of infinitesimals. Then $\mathcal O$ is a valuation ring, namely every element of $N$ or its inverse lies in it. It has $\mathfrak M$ as unique maximal ideal. There is therefore a valuation $v'$ on $N$ such that $v'(x)\geq 0$ on $\mathcal O$ and $v'(x)>0$ on $\mathfrak M$. Renaming the value group, we can set $v'(\epsilon)=1$. Then $v'$ extends the valuation $v$. As $N$ is in the algebraic closure of $\mathbb Q(\epsilon,\pi)$, by standard results on valuation theory (see for example \cite{Prestel}, Theorem 3.2.4), the value group of $v'$ is in the divisible hull of the value group of $v$.

Let $b\in N$ be a positive infinitesimal. By the previous argument $v'(b)$ is rational, so there is $n\in \mathbb N$ such that $v'(b)>v'(\epsilon^{1/n})$. Then $v'(b/(\epsilon^{1/n}))>0$, so $b/(\epsilon^{1/n})$ is infinitesimal and in particular $b<\epsilon^{1/n}$.
\end{proof}

Let $A=\{x\in M : x<\pi\}$. So $A$ is an externally definable initial segment of $M$. Consider the externally definable set $X=\{(x,y)\in M^2 : x\in A \wedge y\notin A\}$. Let $\phi(x,y;t)=(x<t \wedge y>t)$. Then $\phi(x,y;\pi)$ is a definition of $X$. However it is not an honest definition because it is not included in the $M$-definable set $\{(x,y):y-x > \epsilon\}$. We actually show more.

\vspace{5pt}

\underline{Claim 1:} There is no honest definition of $X$ with parameters in $N$.

Proof: Assume that $\chi(x,y)$ is such a definition. Consider $c=\inf\{y-x : y-x>0 \wedge \chi(x,y)\}$. Then $c\in N$. For every $0<\epsilon \in M$ infinitesimal, we have $c>\epsilon$ by the same argument as above. By the previous lemma, there is $0<e\in \mathbb Q$ such that $c>e$. This is absurd as $\chi(x,y) \supseteq X$.

\vspace{5pt}

Let $p$ be the global 1-type such that for $a\in \mathbb M$, $p\vdash x>a$ if and only if there is $b\in A\subset M$ such that $a<b$. Thus $p$ is finitely satisfiable in $M$. Let $a_0=\pi$ and $a_1 \models p|_N$. Consider the formula $\psi(x,y;a_0,a_1)=(x<a_1 \wedge y>a_0)$. 

\vspace{5pt}

\underline{Claim 2:} The formula $\psi$ is an honest definition of $X$.

Proof: Let $\theta(x,y)\in L(M)$ be a definable set. Assume that $X \subseteq \theta(M^2)$ and for a contradiction that $\mathbb M \models (\exists x,y) \psi(x,y;a_0,a_1) \wedge \neg \theta(x,y)$. As $p$ is finitely satisfiable in $M$, there is $u_0 \in M$ such that $\models (\exists x,y) x<u_0 \wedge y>a_0 \wedge \neg \theta(x,y)$. Consider the $M$-definable set $\{v: (\exists x,y) x<u_0 \wedge y>v \wedge \neg \theta(x,y) \}$. By o-minimality, this set has a supremum $m\in M \cup \{+ \infty\}$. We know $m\geq a_0$, so necessarily there is $v_0\in M$, $v_0\notin A$ such that $M \models (\exists x,y) x<u_0 \wedge y>v_0 \wedge \neg \theta(x,y)$. This contradicts the fact that $X\subseteq \theta(M^2)$.

\vspace{5pt}

We therefore see that if $\phi(x,y;a)$ is a formula and $M$ a model, then one cannot in general obtain an honest definition of $\phi(M^2;a)$ with the same parameter $a$. We conjecture that one can find such an honest definition with parameters in a Morley sequence of any coheir of $\tp(a/M)$.

As an application, we give another proof of Shelah's expansion theorem from \cite{Sh783}.

\begin{prop}\label{HonestDefImage}($T$ is $NIP$)
Let $X\subseteq M^k$ be an externally definable set and $f$ an $M$-definable function. Then $f(X)$ is externally definable.
\end{prop}
\begin{proof}
Let $\phi(x,c)$ be an honest definition of $X$. We show that $\theta(y,c)=(\exists x)(\phi(x,c) \wedge f(x)=y)$ is a definition of $f(X)$. First, as $\phi(x,c)$ is a definition of $X$, we have $f(X) \subseteq \theta(M,c)$. Conversely, consider a tuple $a\in M^k\setminus f(X)$. Let $\psi(x)=(f(x)\neq a)$. Then $X \subseteq \psi(M)$. So by definition of an honest definition, $\mathbb M \models \phi(x,c) \rightarrow \psi(x)$. This implies that $\mathbb M \models \neg \theta(a,c)$. Thus $\theta(M,c)\subseteq f(X)$.

In fact one can check that $\theta(y,c)$ is an honest definition of $f(X)$.
\end{proof}

\begin{cor}[Shelah's expansion theorem]\label{ShelahExp}
Let $M\models T$, be $NIP$ and let $M^{Sh}$ denote the expansion of $M$ where we add a predicate for all externally definable sets of $M^k$, for all $k$. Then $M^{Sh}$ has elimination of quantifiers in this language and is $NIP$.
\end{cor}
\begin{proof}
Elimination of quantifiers follows from the previous proposition, taking $f$ to be a projection. As $T$ is $NIP$, it is clear that all quantifier free formulas of $M^{Sh}$ are dependent. It follows that $M^{Sh}$ is dependent.
\end{proof}

Note that there is an asymmetry in the notion of an honest definition. Namely if $\theta(x)$ is an honest definition of some $X\subset M$, then $\neg \theta(x)$ is not in general an honest definition of $M\setminus X$. We do not know about existence of \emph{symmetric} honest definitions which would satisfy this. All we can do is have an honest definition contain one (or indeed finitely many) uniformly definable family of sets. This is the content of the next proposition.

\begin{prop}\label{honest2} ($T$ is $NIP$)
Let $X\subseteq M^k$ be externally definable. Let $\zeta(x,y)\in L$. Define $\Omega=\{ y\in M : \zeta(M,y) \subseteq X\}$. Assume that $\bigcup_{y\in \Omega} \zeta(M,y) = X$.
 
Then there is a formula $\theta(x,y)$ and $d\in \mathbb M$ such that:
\begin{enumerate}
\item $\theta(x,d)$ is an honest definition of $X$,
\item $\mathbb M \models \zeta(x,c) \rightarrow \theta(x,d)$ for every $c\in \Omega$,
\item For any $c_1,..,c_n \in \Omega$, there is $d'\in M$ such that $\theta(M,d') \subseteq X$, and $\zeta(x,c_i) \rightarrow \theta(x,d')$ holds for all $i$.
\end{enumerate}
\end{prop}
\begin{proof}
Let $M\prec N$ where $N$ is $|M|^+$-saturated. Consider the set $Y\subset M$ defined by $$y\in Y \iff (\forall x\in M) (\zeta(x,y) \rightarrow x\in X).$$ By Corollary \ref{ShelahExp}, this is an externally definable subset of $M$, so there is $\psi(x)\in L(N)$ a definition of it. Let also $\phi(x)\in L(N)$ be a definition of $X$. Let $(N,M)\prec (N',M')$ be an elementary extension of the pair, sufficiently saturated. Applying Proposition \ref{HonestLemma} with $p(y)=\{\pred(y)\}$, $A=M$ we obtain a formula $\alpha(y,d) \in L(M')$ such that $\alpha(M,d)=\psi(M)$ and $N'\models \alpha(y,d) \rightarrow^{\pred(y)} \psi(y)$. Set $\theta(x,d)=(\exists y)(\alpha(y,d)\wedge \zeta(x,y))$. We check that $\theta(x,d)$ satisfies the required properties.

First, let $a\in M'$ such that $N'\models \theta(a,d)$. Then as $M'\prec N'$, there is $y_0\in M'$ such that $\alpha(y_0,d) \wedge \zeta(a,y_0)$. By construction of $\alpha(y,d)$, this implies that $N' \models \psi(y_0)$. So by definition of $\psi(y)$, $N'\models \phi(a)$, so $N'\models \theta(x,d) \impl^{\pred(x)} \phi(x)$. Now, assume that $a\in X$. By hypothesis, there is $y_0\in \Omega$ such that $M\models \zeta(a,y_0)$. Then $\psi(y_0)$ holds, and as $y_0\in M$, $N'\models \alpha(y_0,d)$. Therefore $N'\models \theta(a,d)$. This proves that $\theta(x,d)$ is an honest definition of $X$.

Next, if $c\in \Omega$, then $N' \models \alpha(c,d)$, so $N' \models \zeta(x,c) \impl \theta(x,d)$.

Finally, let $c_1,...,c_n \in \Omega$. Then $N' \models (\exists d\in \pred)(\bigwedge \zeta(x,c_i) \impl^{\pred(x)} \theta(x,d))\wedge(\theta(x,d) \impl^{\pred(x)} \phi(x))$. By elementarity, $(N,M)$ also satisfies that formula. This gives us the required $d'$.
\end{proof}

Note in particular that the hypothesis on $\zeta(x,y)$ is always satisfied for $\zeta(x,y)=(x=y)$. As an application, we obtain that large externally definable sets contain infinite definable sets.

\begin{cor}\label{infDefSubset} ($T$ is $NIP$)
Let $X\subseteq M^k$ be externally definable, then if one of the two following conditions is satisfied, $X$ contains an infinite $M$-definable set.
\begin{enumerate}
\item $X$ is infinite and $T$ eliminates the quantifier $\exists^{\infty}$.
\item $|X| \geq \beth_\omega$.
\end{enumerate}
\end{cor}
\begin{proof}
Let $\theta(x,y)$ be the formula given by the previous proposition with $\zeta(x,y)=(x=y)$.
 
If the first assumption holds, then there is $n$ such that for every $d\in M$, if $\theta(M,d)$ has size at least $n$, it is infinite. Take $c_1,...,c_n\in X$ and $d'\in M$ given by the third point of \ref{honest2}. Then $\theta(M,d')$ is an infinite definable set contained in $X$.

Now assume that $|X| \geq \beth_\omega$. By $NIP$, there is $\Delta$ a finite set of formulas and $n$ such that if $(a_i)_{ i<\omega}$ is a $\Delta$-indiscernible sequence and $d\in \mathbb M$, there are at most $n$ indices $i$ for which $\neg(\theta(a_i,d) \leftrightarrow \theta(a_{i+1},d))$. By the Erd\"os-Rado theorem, there is a sequence $( a_i)_{i<\omega_1}$ in $X$ which is $\Delta$-indiscernible. Define $c_i=a_{\omega.i}$ for $i=0,..,n$ and let $d'$ be given by the third point of Proposition \ref{honest2}. Then $\theta(x,d')$ must contain an interval $\langle a_i: \omega\times k \leq i \leq \omega\times {k+1} \rangle$ for some $k\in \{0,..,n-1\}$. In particular it is infinite.
\end{proof}

This property does not hold in general. For example in  the random graph, for any $\kappa$ it is easy to find a model $M$ and $A \subset M$,  $|A| \geq \kappa$ such that every $M$-definable subset of $A$ is finite, while $A$ itself is externally definable.

Also, taking $M=(\mathbb N + \mathbb Z, <)$ and $X=\mathbb N$ shows that $|X|$ has to be bigger than $\aleph_0$ in \ref{infDefSubset} in general.

\begin{question}
Is it possible to replace $\beth_\omega$ by $\aleph_1$ in \ref{infDefSubset}?
\end{question}

\section{On dependent pairs}

\subsubsection*{Setting} In this section, we assume that $T$ is $NIP$. We consider a pair $(M,A)$ with $M\models T$. If $\phi(x,a)$ is some formula of $L_\pred(M)$, then an \emph{honest definition of $\phi(x,a)$ over $A$} is a formula $\theta(x,c)\in L_\pred$, $c\in \pred(\mathbb M)$ such that $\theta(A,c)=\phi(A,a)$ and $\models (\forall x\in \pred)(\theta(x,c)\impl \phi(x,a))$.

(Note that if $M\models T$, $\phi(x,c)\in L(\mathbb M)$ and $X=\phi(M,c)$, then an honest definition of $\phi(x,c)$ over $M$ in the pair $(\mathbb M,M)$ which happens to be an $L$-formula is an honest definition of $X$ in the sense of Definition \ref{honestDef}.)
\\

We say that an $L_{\mathbf{P}}$-formula is \emph{bounded} if it is of the form
$Q_0y_0 \in \mathbf{P} ... Q_ny_n \in \mathbf{P}\, \phi (x,y_0,...,y_n)$
where $Q_i \in \{\exists, \forall \}$ and $\phi(x,\bar y)$ is an $L$-formula, and let $L_{\mathbf{P}}^{bdd}$ be the collection of all bounded formulas.
We say that $T_{\mathbf{P}}$ is bounded if every formula is equivalent to a bounded one. 

Recall that a formula $\phi(x,y)\in L_\pred$ is said to be \emph{$NIP$ over $\pred(x)$} if there is no $L_\pred$-indiscernible (equivalently $L$-indiscernible if $\phi \in L$) sequence $(a_i)_{i<\omega}$ of points of $\pred$ and $y$ such that $\phi(a_i,y) \Leftrightarrow i$ is even. If this is the case, then Proposition \ref{HonestLemma} applies and in particular there is an honest definition of $\phi(x,a)$ over $\pred$ for all $a$.

We say that $T$ (or $T_\pred$) is $NIP$ over $\pred$ if every $L$ (resp. $L_\pred$) formula is.

Given a small subset of the monster $A$ and a set of formulas $\Omega$ (possibly with parameters) we let $A_{ind(\Omega)}$ be the structure with domain $A$ and a relation added for every set of the form $A^n \cap \phi(\bar{x})$, where $\phi(\bar{x}) \in \Omega$.

Notice that $A_{ind(L^{bdd}_{\pred})}$ eliminates quantifiers, while $A_{ind(L)}$ not necessarily does. However $A_{ind(L^{bdd}_{\pred})}$ and $A_{ind(L)}$ are bi-interpretable.
\\
 
\begin{lemme}
\label{HonestDefPred} Assume that $\varphi(xy,c)\in L_{\pred}$
has an honest definition $\vartheta(xy,d)\in L_{\pred}$ over $A$. Then
$\theta(x,d)=(\exists y\in \pred)\vartheta(xy,d)$ is an honest definition
of $\phi(x,c)=(\exists y\in \pred)\varphi(xy,c)$ over $A$.\end{lemme} 
\begin{proof}
For $a \in \pred$, $\theta(a,d)$  $\Rightarrow$   $\vartheta(ab,d)$ for some $b \in \pred$ $\Rightarrow$  $\varphi(ab,c)$ (as $\vartheta(xy,d)$ is honest and $ab \in \pred$) $\Rightarrow$ $\phi(a,c)$. 

For $a \in A$, $\phi(a,c)$   $\Rightarrow$  $\varphi(ab,c)$ for some $b \in A$   $\Rightarrow$  $\vartheta(ab,d)$ (as $\vartheta(A,d)=\varphi(A,c)$)   $\Rightarrow$  $\theta(a,d)$. 
\end{proof}
 
We will be using $\lambda$-big models (see \cite[10.1]{Hod}). We will only use that if $N$ is $\lambda$-big, then it is $ \lambda$-saturated and strongly $\lambda$-homogeneous (that is, for every $\bar{a},\bar{b} \in N^{<\lambda}$ such that $(N,\bar{a}) \equiv (N, \bar{b})$ there is an automorphism of $N$ taking $\bar{a}$ to $\bar{b}$) (see  \cite[10.1.2 + Exercise 10.1.4]{Hod}). Every model $M$ has a  $\lambda$-big elementary extension $N$.  
  
\begin{lemme}
\label{IndSeq} 1) If $N\succeq M$, $M$ is $\omega$-big, $N$ is $|M|^+$-big, and $a,b\in M^{<\omega}$ then $tp_L(a)=tp_L(b) \Leftrightarrow tp_{L_\pred}(a)=tp_{L_\pred}(b)$ in the sense of the pair $(N,M)$.

2) Let $\phi(x,y)\in L_{\pred}$, $(M,A)$  $\omega$-big,
$(a_{i})_{i<\omega}\in M^{\omega}$ be $L_{\pred}$-indiscernible, and
let $\theta(x,d_{0})$ be an honest definition for $\phi(x,a_{0})$ over
$A$ (where $d_0$ is in $\pred$ of the monster model). Then we can find an $L_{\pred}$-indiscernible sequence $(d_{i})_{i<\omega}\in \pred^{\omega}$
such that $\theta(x,d_{i})$ is an honest definition for $\phi(x,a_{i})$
over $A$.\end{lemme}
\begin{proof}
1) We consider here the pair $(N,M)$ as an $L_{\pred}$-structure, where $\pred(x)$ is a new predicate interpreted in the usual way. Let $\sigma\in Aut_{L}(M)$ be such that $\sigma(a)=b$. As $N$ is big, it extends to $\sigma'\in Aut_{L}(N)$, with $\sigma'(M)=M$.
But then actually $\sigma'\in Aut_{L_{\pred}}(N)$ (since it preserves all
$L$-formulas and $\pred$).

2) Let $(N,B)\succeq(M,A)$ be  $|M|^+$-big. We consider the pair of pairs $Th((N,B),(M,A))$ in the language $L_{\pred,\pred'}$,
with $\pred'(N)=M$. By 1) the sequence $(a_{i})_{i<\omega}$ is $L_{\pred,\pred'}$-indiscernible.
The fact that $\theta(x,d_0)$ is an honest definition of $\phi(x,a_0)$ over $A$ is expressible by the formula $$(d_{0}\in \pred)\land  ((\forall x\in \pred'\cap \pred)\,\theta(x,d_{0})\equiv\phi(x,a_{0}))\land((\forall x\in \pred)\theta(x,d_{0})\rightarrow\phi(x,a_{0})).$$
By $L_{\pred,\pred'}$-indiscernibility, for each $i$, we can find $d_{i}$ such that the same formula holds of $(a_i,d_i)$. Then using Ramsey, for any finite $\Delta \subset L_P$, we can find an infinite subsequence $(a_i,d_i)_{i\in I}$, $I\subseteq \omega$ that is $\Delta$-indiscernible. As $(a_i)$ is indiscernible, we can assume $I=\omega$. Then by compactness, we can find the $d_i$'s as required.
\end{proof}

We will need the following technical lemma.

\begin{lemme}
\label{mainLemma} Let $(M,A)\models T_{\pred}$ be $\omega$-big and assume that
$A_{ind(L_{\pred})}$ is $NIP$.

Let $(a_{i})_{i<\omega}\in M^{\omega}$ be $L_{\pred}$-indiscernible,
$(b_{2i})_{i<\omega}\in A^\omega$ and $\Delta((x_{i})_{i<n};(y_{i})_{i<n})\in L_{\pred}$
be such that $\Delta((x_{i})_{i<n};(a_{i})_{i<n})$ has an honest definition
over $A$ by an $L_{\pred}$-formula, and $\models\Delta(b_{2i_{0}},...,b_{2i_{n-1}};a_{2i_{0}},...,a_{2i_{n-1}})$
for any $i_{0},...,i_{n-1}<\omega$.

Then there are $i_{0},...,i_{n-1}\in\omega$ with $i_{j}\equiv j\,(mod\,2)$
and $(b_{i_{j}})_{j\equiv1(mod2),<n}\in \pred$ such that $\models\Delta(b_{i_{0}},...,b_{i_{n-1}};a_{i_{0}},...,a_{i_{n-1}})$.\end{lemme}

\begin{proof} To simplify notation assume that $n$ is even.
Let $$\Delta'((x_{2i})_{2i<n};(y_{i})_{i<n})=(\exists x_{1}x_{3}...x_{n-1}\in \pred)\,\Delta((x_{i})_{i<n};(y_{i})_{i<n}).$$
By assumption and Lemma \ref{HonestDefPred}
$\Delta'((x_{2i})_{2i<n};(a_{i})_{i<n})$ has an honest definition over
$A$ by some $L_{\pred}$-formula, say $\theta((x_{2i})_{2i<n},d)$ with
$d\in \pred$. Since $A_{ind(L_{\pred})}$ is $NIP$, let $N=alt(\theta)$
inside $\pred$.

Choose even $i_{0},i_{2},..., i_{n-2} \in \omega$ such that $i_{j+2}-i_{j}>N$
and consider the sequence $(\bar{a}_{i})_{0<i<N}$ with $\bar{a}_{i}=a_{i_{0}}a_{i_{0}+i}a_{i_{2}}a_{i_{2}+i}...a_{i_{n-2}}a_{i_{n-2}+i}$.
It is $L_{\pred}$-indiscernible (and extends to an infinite $L_{\pred}$-indiscernible sequence). By Lemma \ref{IndSeq}
we can find an $L_{\pred}$-indiscernible sequence $(d_{i})_{i<N}$, $d_i\in \pred$
such that $\theta((x_{2i})_{2i<n};d_{i})$ is an honest definition for
$\Delta'((x_{2i})_{2i<n};\bar{a}_{i})$. By assumption $\theta((b_{i_{2j}})_{2j<n};d_{i})$ holds
for all even $i<N$. But then since $N=alt(\theta)$ inside $\pred$, it must hold for some odd $i'<N$. By honesty this implies
that $\Delta'((b_{i_{2j}})_{2j<n};\bar{a}_{i'})$ holds, and decoding
we find some $(b_{i_{2j}+i'})_{2j<n}\in \pred^{\frac{n}{2}}$ as wanted.\end{proof}

Now the main results of this section.

\begin{thm}\label{depPair1}
Assume $T$ is $NIP$ and $T_{\pred}$ is $NIP$ over $\pred$. Then every bounded formula
is $NIP$.\end{thm}
\begin{proof}
We prove this by induction on adding an existential bounded quantifier (since $NIP$
formulas are preserved by boolean operations). So assume that $\phi(x,y)=(\exists z\in \pred)\,\psi(xz,y)$
has $IP$, where $\psi(xz,y)\in L_{\pred}^{bdd}$ is $NIP$. Then there
is an $\omega$-big $(M,A)\models T_{\pred}$ and an $L_{\pred}$-indiscernible
sequence $(a_{i})_{i<\omega}\in M^{\omega}$ and $c\in M$ such that
$\phi(a_{i},c) \Leftrightarrow i=0 (\text{mod }2)$. Then we can assume
that there are $b_{2i}\in A$ such that $(a_{2i}b_{2i})$ is $L_{\pred}$-indiscernible
and $\models\psi(a_{2i}b_{2i},c)$. 

Notice that from $T_{\pred}$ being $NIP$ over $\pred$ it follows that $A_{ind(L_{\pred})}$
is $NIP$ and that every $L_{\pred}$-formula has an honest definition over
$A$. For $\delta\in L_{\pred}$ take $\Delta_{\delta}((x_{i})_{i<n};(y_{i})_{i<n})$
to be an $L_{\pred}$-formula saying that $(x_{i}y_{i})_{i<n}$ is $\delta$-indiscernible.
Applying Lemma \ref{mainLemma}, we obtain $i_{0},...,i_{n}\in\omega$ with $i_{j}\equiv j\,(\text{mod }\,2)$
and $(b_{i_{j}})_{j\equiv1(\text{mod }2),<n}\in \pred$ such that $(a_{i_k}b_{i_k} )_{k<n}$ is $\delta$-indiscernible. Since $\models \neg(\exists z\in \pred)\psi(a_{2i+1}z,c)$ for all $i$, we see that $\psi(a_{i_k}b_{i_k},c)$ holds if and only if $k$ is even. Taking $n$ and $\delta$ large enough, this contradicts dependence of $\psi(xz,y)$.
\end{proof}

\begin{cor}\label{depPair2}
Assume $T$ is $NIP$, $A_{ind(L)}$ is $NIP$
and $T_{\pred}$ is bounded. Then $T_{\pred}$ is $NIP$.
\end{cor}
\begin{proof}
Since $A_{ind(L^{bdd}_{\pred})}$ is interpretable in $A_{ind(L)}$ the hypothesis implies that $A_{ind(L^{bdd}_{\pred})}$ is $NIP$. Thus, if $\bar{a}=(a_i)_{i<n}$ is a sequence inside $\pred$ then any $\Delta(\bar{x},\bar{a})$ has an honest definition over $A$ (although we don't yet know that $\Delta(\bar{x},\bar{y})$ is $NIP$ over $\pred$, we do know that $\Delta(\bar{x},\bar{a})$ is $NIP$ over $\pred$, so Proposition \ref{HonestLemma} applies). We can then use the same proof as in \ref{depPair1} to ensure that $T_{\pred}$ is $NIP$ over $\pred$, and finally apply Theorem \ref{depPair1} to conclude.
 \end{proof}

\begin{cor}\label{depPair3}
Assume $T$ is $NIP$, and let $(M,N)$ be a pair of models of $T$ ($N\prec M$). Assume that $T_\pred$ is bounded, then $T_\pred$ is $NIP$.
\end{cor}
\begin{proof}
$N_{ind(L)}$ is dependent, and so the hypotheses of Corollary \ref{depPair2} are satisfied.
\end{proof}

Note that the boundedness assumption cannot be dropped, because for example a pair of real closed fields can have $IP$, and also there is a stable theory such that some pair of its models has $IP$ (\cite{Poizat2}).

\section{Applications}

In this section we give some applications of the criteria for the dependence of the pair.

\subsection{Naming an indiscernible sequence}

In \cite{BB} Baldwin and Benedikt prove the following.

\begin{fait}\label{BBResults}($T$ is $NIP$) Let $I \subset M$ be an indiscernible sequence indexed by a dense complete linear order, small in 
$M$ (that is every $p\in S_{< \omega}(I)$ is realised in $M$). Then

1) $Th(M,I)$ is bounded (\cite[Theorem 3.3]{BB}),

2) $(M,I)\equiv(N,J)$ if and only if $EM(I)=EM(J)$ (\cite[Theorem 8.1]{BB}),

3) The $L_{\mathbf{P}}$-induced structure on $\mathbf{P}$ is just the equality (if $I$ is totally transcendental) or the linear order otherwise (\cite[Corollary 3.6]{BB}).
\end{fait}

It is not stated in the paper in exactly this form because the bounded formula from \cite[Theorem 3.3]{BB} involves the order on the indiscernible sequence. However, it is not a problem. If the sequence $I=(a_i)$ is not totally indiscernible, then the order is $L$-definable (maybe after naming finitely many constants). Namely, we will have  $\phi(a_0,...,a_k,a_{k+1},...,a_n) \land \neg \phi(a_0,...,a_{k+1},a_k,...,a_n)$ for some $k<n$ and $\phi \in L$ (as the permutation group is generated by transpositions). But then the order on $I$ is given by $y_1<y_2 \leftrightarrow \phi(a'_0...a'_{k-1},y_1,y_2,a'_{k+2},...,a'_n)$, for any $a'_0...a_{k-1}Ia'_{k+2}...a'_n$ indiscernible (and we can find such $a'_0...a_{k-1}a'_{k+2}...a'_n$ in $M$ by the smallness assumption).
If $I$ is an indiscernible set, then the stable counterpart of their theorem \cite[3.3]{BB} applies giving a bounded formula using just the equality (as the proof in \cite[Section 4]{BB} only uses that for an $NIP$ formula $\phi(x,y)$ and an arbitrary $c$,  $\{a_i : \phi(a_i,c)\}$ is either finite or cofinite, with size bounded by $alt(\phi)$).\\

The following answers Conjecture 9.1 from that paper. 

\begin{prop}\label{BB}
Let $(M,I)$ be a pair as described above, obtained by naming a small, dense, complete indiscernible sequence. Then $T_\pred$ is $NIP$.
\end{prop}
\begin{proof}
By 1) and 3) above, all the assumptions of Corollary \ref{depPair2} are satisfied.
\end{proof}

It also follows that every unstable dependent theory has a dependent expansion with a definable linear order.

Recall the following definition (one of the many equivalent) from \cite{Sh863}.

\begin{defi}\label{strongplus} \cite[Observations 2.1 and 2.10]{Sh863} $T$ is strongly (resp. strongly$^+$) dependent if for any infinite indiscernible sequence $(\bar{a}_{i})_{i \in I}$ with $\bar{a}_i \in \mathbb{M}^{\omega}$, $I$ a complete linear order, and finite tuple $c$ there is a finite $u \subset I$ such that for any two $i_1<i_2 \in u, (i_1,i_2) \cap u = \emptyset$ the sequence $(\bar{a}_i)_{i \in (i_1,i_2)}$ is indiscernible over $c$ (resp. ${c} \cup (\bar{a}_i)_{i \in (- \infty, i_1] \cup [i_2,\infty)}$).

$T$ is $dp$-minimal (resp. $dp^+$-minimal) when for a singleton $c$ there is such a $u$ of size 1.
\end{defi}

For a general $NIP$ theory, the property described in the definition holds, but with $u\subset I$ of size $|T|$, instead of finite. We can take $u$ to be the set of \emph{critical points} of $I$ defined by: $i\in I$ is critical for a formula $\phi(x;y_1,...,y_n,c) \in L$ if there are $j_1,...,j_n \neq i$ such that $\phi(a_i;a_{j_1},...,a_{j_n},c)$ holds, but in every open interval of $I$ containing $i$, we can find some $i'$ such that $\neg \phi(a_{i'};a_{j_1},...,a_{j_n},c)$ holds. One can show (see \cite[Section 3]{Adl}) that given such a formula $\phi(x;y_1,..,y_n,c)$, the set of critical points for $\phi$ is finite. Also $T$ is strongly$^+$ dependent if and only if for every finite set $c$ of parameters, the total number of critical points for formulas in $L(c)$ is finite.
\\

Unsurprisingly $dp$-minimality is not preserved in general after naming an indiscernible sequence. By \cite[Lemma 3.3]{Good} in an ordered $dp$-minimal group, there is no infinite definable nowhere-dense subset, but of course every small indiscernible sequence is like this.

There are strongly dependent theories which are not strongly$^+$ dependent, for example $p$-adics (\cite{Sh863}). In such a theory, strong dependence is not preserved by naming an indiscernible sequence.

\begin{prop}
\label{StrDepNotPreserved} Let $T$ be not strongly $^{+}$ dependent, witnessed by a dense complete indiscernible sequence $(\bar a_i)_{i\in I}$ of finite tuples. Let $\mathbf{P}$ name that sequence in a big saturated model. Then $T_{\mathbf{P}}$ is not strongly dependent.
\end{prop}
\begin{proof}
So let $(\bar a_i)_{i \in I}, c$ witness failure of strong$^+$ dependence. By dependence of $T$,  let $u\subset I$ be chosen as above. Notice that for every $\phi(x;y_1,...,y_n,c)$ , the finite set of its critical points in $I$ is $L_\pred$-definable over $c$ (and possibly finitely many parameters, using order on $I$ in the non-totally indiscernible case, and just the equality otherwise). As in our situation $u$ is infinite, we get infinitely many different finite subsets of $(\bar a_i)_{i\in I}$ definable over $c$, in $T_{\mathbf{P}}$. As $(\bar a_i)_{i\in I}$ is still indiscernible in $T_\pred$ by Fact \ref{BBResults}, 3), this contradicts strong dependence.
\end{proof}

\begin{question}
Is strong$^+$ dependence preserved by naming an indiscernible sequence ?
\end{question}

\subsection{Dense pairs and related structures}

Van den Dries proves in \cite{vdd} that in a dense pair of o-minimal structures, formulas are bounded. This is generalised in \cite{Ber} to lovely pairs of geometric theories of \textthornvari-rank 1. From Theorem \ref{depPair3}, we conclude that such pairs are dependent.

This was already proved by Berenstein, Dolich and Onshuus in \cite{BDO} and generalised by Boxall in \cite{Box}. Our result generalises \cite[Theorem 2.7]{BDO}, since the hypothesis there ($\acl$ is a pregeometry and $A$ is ``innocuous")  imply boundedness of $T_\pred$.  To see this take any two tuples $a$ and $b$ and assume that they have the same bounded types. Let $a' \in \pred$ be such that $aa'$ is a $\pred$-independent tuple. Then by hypothesis, we can find $b'$ such that $tp_{L_{\pred}^{bdd}}(bb')=tp_{L_{\pred}^{bdd}}(aa')$. Now the fact that $aa'$ is $\pred$-independent can be expressed by bounded formulas. In particular $bb'$ is also $\pred$-independent. So by innocuous, $tp_{L_\pred}(aa')=tp_{L_\pred}(bb')$ and we are done.

It is not clear to us if Boxall's hypothesis imply that formulas are bounded. (However, note that in the same paper Boxall applies his theorem to the structure of $\mathbb{R}$ with a named subgroup studied by Belegradek and Zilber, where we know that formulas are bounded.)

The paper \cite{BDO} gives other examples of theories of pairs for which formulas are bounded, including dense pairs of $p$-adic fields and weakly o-minimal theories, recast in the more general setting of \emph{geometric topological structures}.
\\

Similar theorems are proved by G\"unaydin and Hieronymi in \cite{GH1}. Their Theorem 1.3 assumes that formulas are bounded along with other hypothesis, so is included in Theorem \ref{depPair3}. They apply it to show that pairs of the form $(\mathbb R,\Gamma)$ are dependent, where $\Gamma\subset \mathbb R^{>0}$ is a dense subgroup with the \emph{Mann property}. We refer the reader to \cite{GH1} for more details.

In this same paper the authors also consider the case of tame pairs of o-minimal structures. This notion is defined and studied in \cite{DL}. Let $T$ be an o-minimal theory. A pair $(N,M)$ of models of $T$ is \emph{tame} if $M \prec N$ and for every $a\in N$ which is in the convex hull of $M$, there is $\st(a)\in M$ such that $|a-\st(a)|<b$ for every $b\in M^{>0}$. It is proved in \cite{DL} that formulas are bounded is such a pair, so again it follows from Theorem \ref{depPair3} that $T_\pred$ is dependent. Note that G\"unaydin and Hieronymi prove this using their Theorem 1.4 involving quantifier elimination in a language with a new function symbol. This theorem does not seem to factorise trivially through \ref{depPair2}. They also prove in that same paper that the pair $(\mathbb R,2^{\mathbb Z})$ is dependent.

Let $\textsl{C}$ be an elliptic curve over the reals, defined by  $y^2 = x^3 + ax + b$ with $a,b \in \mathbb{Q}$, and let $\mathbf{P} \subseteq \mathbb{Q}^2$ name the set of its rational points. This theory is studied in \cite{GH2}, where it is proved in particular that

\begin{fait}
1) $Th(\mathbb{R},\textsl{C}(\mathbb{Q}))$ is bounded (follows from \cite[Theorem 1.1]{GH2})

2) $A_{ind(L_{\pred})}$ is $NIP$ (follows from \cite[Proposition 3.10]{GH2})
\end{fait}

Applying Corollary \ref{depPair2} we conclude that the pair is dependent.

\bibliography{everything}

\begin{thebibliography}{vdDL95}

\bibitem[Adl08]{Adl}
Hans Adler.
\newblock An introduction to theories without the independence property.
\newblock {\em Preprint}, 2008.

\bibitem[BB00]{BB}
John Baldwin and Michael Benedikt.
\newblock Stability theory, permutations of indiscernibles, and embedded finite
  models.
\newblock {\em Transactions of the American Mathematical Society},
  352(11):4937--4969, 11 2000.

\bibitem[BB04]{BoBa}
Bektur Baizhanov and John Baldwin.
\newblock Local homogeneity.
\newblock {\em Journal of Symbolic Logic}, 69(4):1243--1260, 12 2004.

\bibitem[BDO08]{BDO}
Alexander Berenstein, Alf Dolich, and Alf Onshuus.
\newblock The independence property in generalized dense pairs of structures.
\newblock {\em preprint}, 2008.

\bibitem[Ber]{Ber}
Alexander Berenstein.
\newblock Lovely pairs and dense pairs of o-minimal structures.
\newblock {\em submitted}.

\bibitem[Box09]{Box}
Gareth Boxall.
\newblock {NIP} for some pair-like theories.
\newblock {\em preprint}, 2009.

\bibitem[BP98]{Poizat}
Yerzhan Baisalov and Bruno Poizat.
\newblock Paires de structures o-minimales.
\newblock {\em Journal of Symbolic Logic}, 63:570--578, 1998.

\bibitem[CZ01]{CaZi}
Enrique Casanovas and Martin Ziegler.
\newblock Stable theories with a new preicate.
\newblock {\em Journal of symbolic logic}, 66(3):1127--1140, 09 2001.

\bibitem[EP05]{Prestel}
Antonio~J. Engler and Alexander Prestel.
\newblock {\em Valued fields}.
\newblock Springer Monographs in Mathematics. Springer-Verlag, Berlin, 2005.

\bibitem[GH09]{GH2}
Ayhan G\"unaydin and Philipp Hieronymi.
\newblock The real field with the rational points of an elliptic curve.
\newblock {\em arXiv:0906.0528}, 2009.

\bibitem[GH10]{GH1}
Ayhan G\"unaydin and Philipp Hieronymi.
\newblock Dependent pairs.
\newblock {\em MODNET preprint 146}, 2010.

\bibitem[Goo09]{Good}
John Goodrick.
\newblock A monotonicity theorem for dp-minimal densely ordered groups.
\newblock {\em Journal of Symbolic Logic, accepted}, 2009.

\bibitem[Gui09]{Guingona}
Vincent Guingona.
\newblock Dependence and isolated extensions.
\newblock {\em preprint}, 2009.

\bibitem[Hod93]{Hod}
Wilfrid Hodges.
\newblock {\em Model Theory}, volume~42 of {\em Encyclopedia of mathematics and
  its applications}.
\newblock Cambridge University Press, Great Britain, 1993.

\bibitem[OP07]{OnPe}
Alf Onshuus and Ya'acov Peterzil.
\newblock A note on stable sets, groups, and theories with nip.
\newblock {\em Mathematical Logic Quarterly}, 53:295--300, 2007.

\bibitem[Pil07]{Pillayexp}
Anand Pillay.
\newblock On externally definable sets and a theorem of {S}helah.
\newblock {\em Felgner Festchrift,Studies in Logic, College Publications},
  2007.

\bibitem[Poi83]{Poizat2}
Bruno Poizat.
\newblock Paires de structures stables.
\newblock {\em The Journal of Symbolic Logic}, 48:239--249, (1983.

\bibitem[She04]{Sh783}
Saharon Shelah.
\newblock Dependent first order theories, continued.
\newblock {\em arXiv:math/0406440v1}, 2004.

\bibitem[She05]{Sh863}
Saharon Shelah.
\newblock Strongly dependent theories.
\newblock {\em arXiv:math.LO/0504197}, 2005.

\bibitem[vdD98]{vdd}
Lou van~den Dries.
\newblock Dense pairs of o-minimal structures.
\newblock {\em Fund. Math.}, 157:61--78, 1998.

\bibitem[vdDL95]{DL}
Lou van~den Dries and Adam~H. Lewenberg.
\newblock T-convexity and tame extensions.
\newblock {\em Journal of Symbolic Logic}, 155(3):807--836, 1995.

\end{thebibliography}

\end{document}